\documentclass[11pt, one sided]{amsart}
\RequirePackage[colorlinks,citecolor=blue,urlcolor=blue]{hyperref}

\usepackage{latexsym,amssymb,amsmath,amsfonts,amsthm}
\usepackage{amsthm,amsmath}
\usepackage{float}
\usepackage{enumerate}
\usepackage[margin=3cm]{geometry}
\usepackage{bbm}
\usepackage{subfigure}
\usepackage{graphicx}
\usepackage[toc,page]{appendix}
\usepackage{multirow}
\usepackage{amsfonts}
\usepackage[all]{xy}
\usepackage{caption}
 \usepackage{geometry}                
\usepackage{graphicx}
\usepackage{amssymb}
\usepackage{epstopdf}
\usepackage{color}
\usepackage{bbm, dsfont}
\usepackage{hyperref}
\usepackage[utf8]{inputenc}
\usepackage{amssymb,amsthm,amsmath,amssymb,wrapfig,dsfont}
\usepackage[dvipsnames]{xcolor}
\definecolor{myred}{RGB}{251,154,133}
\definecolor{myblue}{RGB}{153,206,227}
\definecolor{mylightblue}{RGB}{0, 150, 255}
\definecolor{mygreen}{RGB}{32, 210, 64}
\definecolor{mygray}{RGB}{220, 220, 220}

\usepackage{tikz}
\usetikzlibrary{decorations.pathmorphing}
\tikzset{snake it/.style={decorate, decoration=snake}}
\usetikzlibrary{shapes.geometric,positioning,decorations.pathreplacing} 
\usepackage{pgfplots}

\newtheorem{theorem}{Theorem}
\newtheorem{definition}{Definition}[section]
\newtheorem{lemma}{Lemma}
\newtheorem{remark}{Remark}[section]

\newtheorem{corollary}{Corollary}[section]

\def\beq{ \begin{equation} }
\def\eeq{ \end{equation} }

\def\square{\vcenter{\vbox{\hrule height .4pt
  \hbox{\vrule width .4pt height 5pt \kern 5pt
        \vrule width .4pt} \hrule height .4pt}}}

\newcommand{\BN}{{\mathbb{N}}}

\newcommand{\BR}{{\mathbb{R}}}

\newcommand{\ind}{{\mathbbm{1}}}

\newcommand{\bae}{\begin{equation}\begin{aligned}}
\newcommand{\eae}{\end{aligned}\end{equation}}

\DeclareFontFamily{OML}{rsfs}{\skewchar\font'177}
\DeclareFontShape{OML}{rsfs}{m}{n}{ <5> <6> rsfs5 <7> <8> <9>
rsfs7 <10> <10.95> <12> <14.4> <17.28> <20.74> <24.88> rsfs10 }{}
\DeclareMathAlphabet{\mathfs}{OML}{rsfs}{m}{n}

\newcommand{\rdb}{\rho_{\text{DB}}}

\begin{document}

\title{An elementary proof for the Double Bubble problem in $\ell^1$ norm}

\author{Parker Duncan}
\address[Parker Duncan]{Department of mathematics, Texas A\&M University}
\email{parkeraduncan@tamu.edu}

\author{Rory O'Dwyer}
\address[Rory O'Dwyer\footnote{Work on this paper was done while the second author was affiliated with Texas A\&M University.}]{Department of Physics, Stanford University}
\email{rodwyer100@tamu.edu }

\author{Eviatar B. Procaccia}
\address[Eviatar B. Procaccia\footnote{Research supported by NSF grant 1812009.}]{Faculty of Industrial Engineering and Management, Technion - Israel Institute of Technology and Department of mathematics, Texas A\&M University}
\urladdr{http://www.math.tamu.edu/~procaccia}
\email{eviatarp@technion.ac.il}

\maketitle
\begin{abstract}
We study the double bubble problem with perimeter taken with respect to the $\ell_1$ norm on $\BR^2$. We give an elementary proof for the existence of minimizing sets for any volume ratio parameter $0<\alpha\le1$ by direct comparison to a small family of parameterized sets. By simple analysis on this family we obtain the minimizing shapes found in \cite{morgan1998wulff}.
\end{abstract}



\section{Introduction}
In \cite{foisy1993standard} and \cite{doublebubbleconj} the double bubble conjecture in $\BR^2$ and $\BR^3$ was established, stating that the unique perimeter-minimizing double bubble which encloses two fixed volumes consists of three spherical caps whose tangents meet at an angle of 120 degrees. Recently the Gaussian double bubble conjecture was established by Milman and Neeman \cite{milman2018gaussian}. These problems are an extension of the classical isoperimetric problem stating that the perimeter minimizing shape for a fixed volume is the sphere. The case of the isoperimetric problem in which the perimeter was taken with respect to any norm on $\BR^n$ was solved as well. Namely, Taylor \cite{taylor1974existence,taylor1975unique} proved that the unique solution of the isoperimetric inequality with respect to any norm $\rho$ is the renormalized ball in the dual norm, the so called Wulff construction \cite{wul1901frage}. For example the isoperimetric shape with respect to the $\ell_1$ is the $\ell_\infty$ ball $[0,1]^n$. Such non isotropic isoperimetric problems arose naturally in the field of probability, mostly in scaling limits of percolation clusters in a lattice \cite{alexander1990wulff, biskup2015isoperimetry, bodineau2000rigorous, cerf2006wulff}. In this paper we study the double bubble problem with respect to the $\ell_1$ norm. The result discussed in this paper was first proved in \cite{morgan1998wulff}, and is based on previous geometric measure theory results. However our proof is self contained and considerably simpler. Moreover our simple approach, that uses no geometric measure theory, seems to be more amenable to generalizations to higher dimensions.  

\subsection{Notations and results}
For any Lebesgue-measurable set $A\subset\mathbb{R}^2$, let $\mu(A)$ be its Lebesgue measure.  For a simple curve $\lambda:[a,b]\rightarrow\mathbb{R}^2$, not necessarily closed, where $\lambda(t)=(x(t),y(t))$, define its $\ell^1$ length by 
$$\rho(\lambda)=\sup_{N\geq1}\sup_{a\leq t_1\leq...\leq t_N\leq b}\sum_{i=1}^N\big(\left|x(t_{i+1})-x(t_i)\big|+\big|y(t_{i+1})-y(t_i)\big|\right).$$  
If we wish to measure only a portion of the curve $\lambda$, it will be denoted $\rho(\lambda
([t,t']))$, where $[t,t']\subset[a,b]$.  For simplicity we assume that $[a,b]=[0,1]$ unless otherwise stated.  \\

We say that two curves $\lambda,\lambda':[0,1]\rightarrow\mathbb{R}^2$ intersect nontrivially if there are intervals $[s,s'],[t,t']\subset[0,1]$ such that $\lambda([s,s'])=\lambda'([t,t'])$, then their nontrivial intersection can be written as the union of curves $\lambda_i$ such that $\lambda_i([0,1])=\lambda([s_i,s_{i+1}])\rightarrow\mathbb{R}^2$ for some intervals $[s_i,s_{i+1}]$, and we define the length of the nontrivial intersection to be $\rho(\lambda\cap\lambda'):=\sum_i\rho(\lambda_i)$.  \\

Here we are interested in the double bubble perimeter of two simply connected open sets $A,B\subset\mathbb{R}^2$ where the boundary of $A$, $\partial A$, is a closed, simple, rectifiable curve, and similarly for $B$, and where the intersection of the boundaries of $A$ and $B$ is a union of disjoint, rectifiable curves.  The double bubble perimeter is defined as $$\rho_{\text{DB}}(A,B)=\rho(\lambda)+\rho(\lambda')-\rho(\lambda\cap\lambda'),$$ where $\partial A=\lambda([0,1])$, and $\partial B=\lambda'([0,1])$.  We will also use the notation $\rho(\lambda)=\rho(\partial A)$.

For $\alpha\in(0,1]$, define:
$\gamma_{\alpha}=\{(A,B):A,B\subset\mathbb{R}^2$, where $ A,B$ are disjoint, simply connected open sets, and $\partial A$,$\partial B$,$\partial A\cap \partial B$ are unions of closed, continuous, simple, rectifiable curves, with $\mu(A)=1,\mu(B)=\alpha\}$.  

Let $$\rho_{\text{DB}}(\Gamma_{\alpha}):=\inf\{\rho_{\text{DB}}(A,B):(A,B)\in\gamma_{\alpha}\},$$ be the infimum of the double bubble perimeter (bounded below by zero).

The main result in this paper is:

\begin{theorem}\label{thm:maintheorem}

For $0<\alpha\le 1$, 
\begin{enumerate}[I.]
\item \label{existencepart}The set $\Gamma_{\alpha}:=\{(A,B)\in\gamma_{\alpha}: \rho_{\text{DB}}(A,B)=\rho_{\text{DB}}(\Gamma_{\alpha})\}$ is not empty.  

\item \label{part1}The infimum,
$$\rho_{\text{DB}}(\Gamma_{\alpha})=(4\sqrt{1+\alpha}+2\sqrt{\alpha})\ind_{(0,\frac{688-480\sqrt{2}}{49}]}+(4+2\sqrt{2\alpha})\ind_{\left(\frac{688-480\sqrt{2}}{49},\frac{1}{2}\right)}+(2\sqrt{6(1+\alpha)})\ind_{[\frac{1}{2},1]}.$$

\item \label{part:exampleshapes} For $\alpha=\frac{688-480\sqrt{2}}{49}$ we have $|\Gamma_{\alpha}|\ge 2$, with $\Gamma_\alpha$ containing both sets in Figure \ref{fig:possibleinf} (a) and (b). Moreover for $\alpha\in[1/2,1]$, $\Gamma_\alpha$ contains Figure \ref{fig:possibleinf} (c), for $\alpha\in\left(\frac{688-480\sqrt{2}}{49},\frac{1}{2}\right)$, $\Gamma_\alpha$ contains Figure \ref{fig:possibleinf} (b), and for $\alpha\in\left(0,\frac{688-480\sqrt{2}}{49}\right)$, $\Gamma_\alpha$ contains Figure \ref{fig:possibleinf} (a).
\begin{figure}[h!]\begin{tikzpicture}[scale=0.7]

\draw[blue,  thin] (0,4) to (4,4);
\draw[blue,  thin] (0,0) to (4,0);
\draw[blue,  thin] (0,0) to (0,4);
\draw[blue,  thin] (4,0) to (4,4);
\draw[blue, thin] (3,3) to (3,4);
\draw[blue, thin] (3,3) to (4,3);
\draw (2,0.32) node{$\sqrt{1+\alpha}$};
\draw (0.9,2) node{$\sqrt{1+\alpha}$};
\draw (2.5,3.5) node{$\sqrt{\alpha}$};
\draw (3.5,2.5) node{$\sqrt{\alpha}$};
\draw (2,-0.5) node{(a)};
\draw[blue,  thin] (5,4) to (9,4);
\draw[blue,  thin] (5,0) to (9,0);
\draw[blue,  thin] (5,0) to (5,4);
\draw[blue,  thin] (9,0) to (9,4);
\draw (7,0.25) node{$1$};
\draw (5.25,2) node{$1$};
\draw (7,-0.5) node{(b)};
\draw[blue,  thin] (9,1) to (10.5,1);
\draw[blue,  thin] (9,3) to (10.5,3);
\draw[blue,  thin] (10.5,1) to (10.5,3);
\draw (9.80,1.75) node{$\sqrt{2\alpha}$};
\draw (9.80,3.5) node{$\frac{\sqrt{2\alpha}}{2}$};
\draw[blue,  thin] (11,4) to (14.5,4);
\draw[blue,  thin] (11,0) to (14.5,0);
\draw[blue,  thin] (11,0) to (11,4);
\draw[blue,  thin] (14,0) to (14,4);
\draw[blue,  thin] (14.5,4) to (16.5,4);
\draw[blue,  thin] (14.5,0) to (16.5,0);
\draw[blue,  thin] (16.5,4) to (16.5,0);
\draw (12.05,1.8) node{$\sqrt{\frac{2(1+\alpha)}{3}}$};
\draw (12.4,3.4) node{$\sqrt{\frac{3}{2(1+\alpha)}}$};
\draw (15.25,3.4) node{$\alpha\sqrt{\frac{3}{2(1+\alpha)}}$};
\draw (14,-0.5) node{(c)};

\end{tikzpicture}\caption{\label{fig:possibleinf}}
\end{figure}
\end{enumerate}

\end{theorem}

Immediately from Theorem \ref{thm:maintheorem} part \ref{part1} we get:
\begin{corollary} 

There are two critical $\alpha$'s at which $\rho_{\text{DB}}(\Gamma_{\alpha})$ undergoes a phase transition. The first, at $\alpha=\frac{688-480\sqrt{2}}{49}$, is discontinuous in the first order derivative, while the second, at $\alpha=1/2$, is discontinuous in the second order derivative.

\end{corollary}

Before we explain the proof strategy we define in Figure \ref{fig:falpha} a finite family of set types abbreviated $\mathfs{F}_\alpha\subset\gamma_\alpha$: 
\begin{figure}[h!]
\begin{tikzpicture}[scale=0.4]

\draw[blue,    thick] (0,4) to (3,4);
\draw (1.5,3.5) node {$a$};
\draw[blue,    thick] (0,0) to (3,0);
\draw (.5,2) node{$b$};
\draw[blue,    thick] (0,0) to (0,4);
\draw (4.5,2.6) node{$d$};
\draw[blue,    thick] (3,0) to (3,4);
\draw (4,3.5) node{$c$};
\draw (3,7.25) node{${\tiny\begin{array}{l} \textrm{Kissing rectangles} \\\rdb=2(a+b+c)+d\\ ab=\alpha\textrm{ or 1}\\ cd=\textrm{1 or }\alpha \\b\ge d\\ a,b,c,d>0 \end{array}}$};


\draw[red,   thick] (3,1.5) to (5,1.5);
\draw[red,   thick] (5,1.5) to (5,4);
\draw[red,   thick] (5,4) to (3,4);


\draw[blue,   thick] (7,4) to (10,4); 
\draw (7.5,2) node{$a$};
\draw[blue,   thick] (7,0) to (10,0);
\draw (8.5,.5) node{$b$};
\draw[blue,   thick] (7,0) to (7,4);
\draw (11.5,2) node{$c$};
\draw[blue,   thick] (10,0) to (10,4);
\draw (11,4.5) node{$d$};
\draw[red,   thick] (10,0) to (12,0);
\draw[red,   thick] (12,0) to (12,5); 
\draw[red,   thick] (12,5) to (7,5);
\draw[red,   thick] (7,5) to (7,4);
\draw (11,7.25) node{${\tiny \begin{array}{l} \textrm{Embedded rectangle}\\\rdb=2(c+d)+a+b\\ ab=\alpha\textrm{ or 1}\\ c\ge a,b\ge d\\ cd-ab=\textrm{1 or }\alpha\\ a,b,c,d>0 \end{array}}$};
\draw[blue,   thick] (15,4) to (18,4); 
\draw (15.5,2) node{$a$};
\draw[blue,   thick] (15,0) to (18,0);
\draw (16.5,.5) node{$b$};
\draw[blue,   thick] (15,0) to (15,4);
\draw (19.25,1.5) node{$f$};
\draw[blue,   thick] (18,0) to (18,4);
\draw (16.5,4.5) node{$c$};
\draw (17,3.65) node{$d$};
\draw (17.7,2.5) node{$e$};
\draw[red,   thick] (18,1) to (20,1);
\draw[red,   thick] (20,1) to (20,5); 
\draw[red,   thick] (20,5) to (16,5);
\draw[red,   thick] (16,5) to (16,4);
\draw (21,7.25) node{${\tiny \begin{array}{l} \textrm{General case}\\\rdb=2(a+b+c+f)+(d+e)\\ ab=\textrm{1 or }\alpha\\ b\ge d,a\ge e\\ cd+cf+ef=\alpha\textrm{ or 1}\\ a,b,c,d,e,f>0 \end{array}}$};

\end{tikzpicture}\caption{\label{fig:falpha}}
\end{figure}

While the general case encapsulates the other cases, the kissing rectangles and embedded rectangle cases are important enough to annotate and include with names. We will be referring to these annotations later in the paper.


$~$

The strategy for proving Theorem \ref{thm:maintheorem} follows 3 steps:
\begin{enumerate}
\item Begin with any two sets $(A,B)\in \gamma_{\alpha}$, and find sets $(\tilde{A},\tilde{B})\in\mathfs{F}_{\alpha}$ with $\rho_{\text{DB}}(\tilde{A},\tilde{B})\le \rho_{\text{DB}}(A,B)$. This part is done in Section \ref{sec:optimalsets}.
\item Since the sets in $\mathfs{F}_\alpha$ are very simple to analyze, and the family is finite, we can show the existence of $$\arg\inf\{\rho_{\text{DB}}({A},{B}):(A,B)\in\mathfs{F}_\alpha\}.$$ This part is done in Section \ref{sec:KKT}.

\item Finally by the previous points, these sets achieve the infimum over all of $\gamma_\alpha$ proving the existence of an element in $\Gamma_\alpha$. Moreover we get the phase transitions in $\alpha$ and show non uniqueness for the first phase transition. This is done in Section \ref{sec:proof}.
\end{enumerate}

\section{Finding the sets in $\mathfs{F}(\alpha)$}\label{sec:optimalsets}
Our goal in this section is to find elements of $\mathfs{F}_\alpha$ with a smaller double bubble perimeter than given sets $A$ with $\mu(A)=1$, and $B$ with $\mu(B)=\alpha$.  

\begin{definition}\label{def:aboxnotations}
\bae
&A^{\square}:=[a_{{\rm left}},a_{{\rm right}}]\times[a_{{\rm bottom}},a_{{\rm top}}]\text{, where} \\
&a_{{\rm left}}=\inf\{x:(x,y)\in A\text{ for some }y\in\mathbb{R}\}\\
&a_{{\rm right}}=\sup\{x:(x,y)\in A\text{ for some }y\in\mathbb{R}\}\\
&a_{{\rm bottom}}=\inf\{y:(x,y)\in A\text{ for some }x\in\mathbb{R}\}\\
&a_{{\rm top}}=\sup\{y:(x,y)\in A\text{ for some }x\in\mathbb{R}\}
\eae
\end{definition}


\begin{lemma} \label{lem:squareboundary}
$\rho(A^{\square})\leq\rho(A)$ and $\mu(A)\leq\mu(A^{\square})$. 
\end{lemma}

\begin{proof}  By definition, $A\subset A^{\square}$.  Therefore, by monotonicity of Lebesgue measure, $\mu(A)\leq\mu(A^{\square})$.

Now we show that $\rho(A^{\square})\leq\rho(A)$.  Let $H_1$ be the horizontal line passing through $a_{\rm top}$, and $H_2$ be the horizontal line passing through $a_{\rm bottom}$.  Similarly define $V_1$ and $V_2$ to be the vertical lines passing through $a_{\rm left}$ and $a_{\rm right}$, respectively.  Let $D_{H_1,H_2}$be the distance from $H_1$ to $H_2$, and $D_{V_1,V_2}$ be the distance between $V_1$ and $V_2$.  $\partial A$ must touch $H_1$ in at least one point, say $p_1$, and similarly must touch $H_2$ in at least one point, say $p_2$. See Figure \ref{fig:contained2} for an illustration of the notations.

\begin{figure}[H]
\begin{tikzpicture}[scale=0.80]

\node at (3.4,1.4) {$\textcolor{blue}{A^{\square}}$};
\node at (1.9,2.2) {$\textcolor{red}{A}$};
\node at (1.8,3.2) {$\textcolor{red}{\partial A}$};
\node at (4.3,4.2) {$p_1$};
\node at (1.5,0.8) {$p_2$};
\node at (3,2.2) {$\lambda_1$};
\node at (1.4,1.8) {$\lambda_2$};

\draw[blue, thick] (1,1) to (1,4);
\draw[blue, thick] (1,4) to (4,4);
\draw[blue, thick] (4,4) to (4,1);
\draw[blue, thick] (4,1) to (1,1);

\draw [red,thick] plot [smooth, tension=0.5] coordinates {(4,4) (1.5,2.6)(1,2.5) (1.4,1)(2.5,2)(4,4)};




\draw[dotted, thick] (-1,4) to (6,4);
\node at (5.4,4.3) {$H_1$};


\draw[dotted, thick] (-1,1) to (6,1);
\node at (5.4,1.3) {$H_2$};



\draw[<->,thick] (-0.9,1.02) to (-0.9,3.98);
\node at (-1.6,2.5) {$D_{H_1,H_2}$};



\draw[dotted, thick] (1,-1) to (1,6);
\node at (1.3,5.8) {$V_1$};


\draw[dotted, thick] (4,-1) to (4,6);
\node at (4.3,5.8) {$V_2$};

\draw[<->,thick] (1.02,-0.98) to (3.98,-0.98);
\node at (2.5,-1.3) {$D_{V_1,V_2}$};


\end{tikzpicture}
\caption{~\label{fig:contained2}}
\end{figure}

Since $A$ is open and $\partial A$ is simple, there must be at least two disjoint paths in $\partial A$ from $p_1$ to $p_2$, abbreviate them $\lambda_i(t)=(x_i(t),y_i(t)):[0,1]\to\BR^2$, for $i\in\{1,2\}$.  The vertical portion of these paths must be at least $D_{H_1,H_2}$ i.e. for any $0\le t_1\le t_2\le\cdots\le t_n\le 1$, and $i\in\{1,2\}$ we have that
$$
\sum_{j=1}^N\big(\left|y_i(t_{j+1})-y_i(t_j)\big|\right)\ge D_{H_1,H_2}$$
Similarly there must be at least two disjoint paths from a point on $V_1$ to a point on $V_2$, whose horizontal distances must be at least $D_{V_1,V_2}$.  We have so far found that the boundary of $A$ must measure at least $2\cdot D_{V_1,V_2}+2\cdot D_{H_1,H_2}$, which is the length of the boundary of $A^{\square}$.  That is, $\rho(A^{\square})\leq\rho(A)$, as claimed.

\end{proof}

\begin{lemma}\label{lem:containedrec}
If $(A,B)\in\gamma_{\alpha}$, and $B^{\square}\subset A^{\square}$, then there exists $(\tilde{A},\tilde{B})\in\mathfs{F}_\alpha$, such that $\rdb(\tilde{A},\tilde{B})\leq\rdb(A,B)$, and $(\tilde{A},\tilde{B}$) are either kissing rectangles or embedded rectangles. 

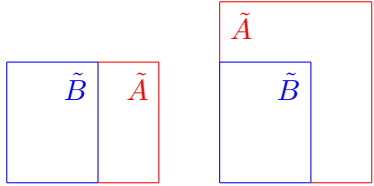
\begin{center}
\begin{tikzpicture}[scale=0.4]

\draw[blue,  thin] (0,4) to (3,4) node[anchor=north east]{$\tilde{B}$};
\draw[blue,  thin] (0,0) to (3,0);
\draw[blue,  thin] (0,0) to (0,4);
\draw[blue,  thin] (3,0) to (3,4);

\draw[red, thin] (3,0) to (5,0);
\draw[red, thin] (5,0) to (5,4) node[anchor=north east]{$\tilde{A}$};
\draw[red, thin] (5,4) to (3,4);


\draw[blue,  thin] (7,4) to (10,4) node[anchor=north east]{$\tilde{B}$};
\draw[blue,  thin] (7,0) to (10,0);
\draw[blue,  thin] (7,0) to (7,4);
\draw[blue,  thin] (10,0) to (10,4);

\draw[red, thin] (10,0) to (12,0);
\draw[red, thin] (12,0) to (12,6);
\draw[red, thin] (12,6) to (7,6) node[anchor=north west]{$\tilde{A}$};
\draw[red, thin] (7,6) to (7,4);

\end{tikzpicture}

\end{center}

\end{lemma}

\begin{proof}

If $B^{\square}$ is contained in $A^{\square}$ there are several options to consider, namely that one, two, three, or no edges of $B^{\square}$ could touch the same number of edges in $A^{\square}$.  Let's take the instance when none of the edges of $B^{\square}$ touch any of the edges of $A^{\square}$, that is, $b_{\rm top}<a_{\rm top}$, $b_{\rm right}<a_{\rm right}$, $a_{\rm left}<b_{\rm left}$, and $a_{\rm bottom}<b_{\rm bottom}$. From this case we can easily derive the results for the other cases. Let $H_1=\mathbb{R}\times\{a_{\rm top}\}$, $H_2=\mathbb{R}\times\{b_{\rm top}\}$, $H_3=\mathbb{R}\times\{b_{\rm bottom}\}$, and $H_4=\mathbb{R}\times\{a_{\rm bottom}\}$.  Now we define the distance between $H_1$ and $H_2$ to be $D_{H_1,H_2}=a_{\rm top}-b_{\rm top}$, the distance between $H_2$ and $H_3$ to be $D_{H_2,H_3}=b_{\rm top}-b_{\rm bottom}$, and the distance between $H_3$ and $H_4$ to be $D_{H_3,H_4}=b_{\rm bottom}-a_{\rm bottom}$.  Similarly we call the vertical line through $a_{\rm left}$ $V_1$, the vertical line through $b_{\rm left}$ $V_{2}$, the vertical line through $b_{\rm right}$ $V_3$, and the vertical line through $a_{\rm right}$ $V_4$, naming the distances between these lines just as before.  See Figure \ref{fig:contained} for illustration. 

\begin{figure}[H]
\begin{tikzpicture}[scale=0.80]

\node at (5.5,5.5) {$A^{\square}$};
\draw[red, thick] (0,0) to (0,6);
\draw[red, thick] (0,6) to (6,6);
\draw[red, thick] (6,6) to (6,0);
\draw[red, thick] (6,0) to (0,0);

\node at (1.4,1.4) {$B^{\square}$};
\draw[blue, thick] (1,1) to (1,4);
\draw[blue, thick] (1,4) to (4,4);
\draw[blue, thick] (4,4) to (4,1);
\draw[blue, thick] (4,1) to (1,1);


\draw[dotted, thick] (-3,6) to (9,6);
\node at (8.4,6.3) {$H_1$};


\draw[dotted, thick] (-3,4) to (9,4);
\node at (8.4,4.3) {$H_2$};


\draw[dotted, thick] (-3,1) to (9,1);
\node at (8.4,1.3) {$H_3$};


\draw[dotted, thick] (-3,0) to (9,0);
\node at (8.4,0.3) {$H_4$};


\draw[<->,thick] (-2.9,4.02) to (-2.9,5.98);
\node at (-3.6,5) {$D_{H_1,H_2}$};


\draw[<->,thick] (-2.9,1.02) to (-2.9,3.98);
\node at (-3.6,2.5) {$D_{H_2,H_3}$};


\draw[<->,thick] (-2.9,0.02) to (-2.9,0.98);
\node at (-3.6,0.5) {$D_{H_3,H_4}$};


\draw[dotted, thick] (0,-2) to (0,8);
\node at (0.3,7.8) {$V_1$};


\draw[dotted, thick] (1,-2) to (1,8);
\node at (1.3,7.8) {$V_2$};


\draw[dotted, thick] (4,-2) to (4,8);
\node at (4.3,7.8) {$V_3$};


\draw[dotted, thick] (6,-2) to (6,8);
\node at (6.3,7.8) {$V_4$};


\draw[<->,thick] (0.02,-1.98) to (0.98,-1.98);
\node at (0.5,-2.3) {$D_{V_1,V_2}$};


\draw[<->,thick] (1.02,-1.98) to (3.98,-1.98);
\node at (2.5,-2.3) {$D_{V_2,V_3}$};


\draw[<->,thick] (4.02,-1.98) to (5.98,-1.98);
\node at (5,-2.3) {$D_{V_3,V_4}$};

\end{tikzpicture}
\caption{~\label{fig:contained}}
\end{figure}

Now, in $\partial A\cup\partial B$ we need to find three paths with vertical lengths of $D_{H_2,H_3}$, and in $\partial A$ two paths with vertical lengths of $D_{H_1,H_2}$ and the same for $D_{H_3,H_4}$.  Further, we need these paths to be disjoint so we don't count anything more than once.  Then we would do the analogous process for horizontal distances.  First, we find points $p_1\in H_1\cap\partial A$, and $p_2\in H_4\cap\partial A$.  There must be two distinct paths in $\partial A$ between these two points, which we call $\lambda_i(t)=(x_i(t),y_i(t)):[0,1]\rightarrow\mathbb{R}^2$, for $i\in\{1,2\}$, both of which must have vertical distance of at least $D_{H_1,H_2}+D_{H_2,H_3}+D_{H_3,H_4}$. That is, for any $0\leq t_1\leq t_2\leq...\leq t_N\leq1$, we have for $i\in\{1,2\}$, 

$$\sum_{j=1}^N\left(|y_i(t_{j+1})-y_i(t_j)|\right)\geq D_{H_1,H_2}+D_{H_2,H_3}+D_{H_3,H_4}.$$

To complete our search for enough vertical length, it remains to find one final path with vertical length $D_{H_2,H_3}$ that we have yet to count among these crossings. For a path $\gamma:[0,1]\to \BR^2$ we say that a subpath $\gamma[a,b]$ is a crossing of $S:=\mathbb{R}\times(b_{\rm bottom},b_{\rm top})$ if $\gamma(a)\in H_2$,  $\gamma(b)\in H_3$ and for all $a<t<b$, $\gamma(t)\notin \{H_2,H_3\}$ (or in the other direction). If $\partial A$ contains more than two crossings, then each of these crossings must have vertical length at least $D_{H_2,H_3}$, and we are done. So we may assume that $\partial A$ only contains two crossings of $S$. We denote these two crossings $\xi_1,\xi_2$.  
Since $\partial A$ contains exactly two crossings of $S$, $S\setminus \xi_1\cup\xi_2$ consists of three open sets, exactly two of which are unbounded and have joint boundary with both $H_2$ and $H_3$.  We call these two open sets $S_1$ and $S_2$.  Since $B$ is open and connected and contained in $S$, $B$ must be contained in one of these unbounded open sets, say w.l.o.g. $S_2$ (note that if there are more than 2 crossings then $B$ can be contained in a bounded set). Since $B$ is open and $\partial B$ is rectifiable, there are at least two distinct paths in $\partial B$ from $H_2$ to $H_3$, which we call $\lambda_3=(x_3(t),y_3(t))$, and $\lambda_4=(x_4(t),y_4(t))$.  Both of these paths must be in $\overline{S_2}$. By planarity only one of these paths might intersect $\xi_2$ say w.l.o.g. $\lambda_3$.  This means that we have not counted the vertical part of $\lambda_4$, and it must have vertical length at least $D_{H_2,H_3}$.  That is to say, that for any $0\leq t_1\leq t_2\leq...\leq t_N\leq1$, we have 

$$\sum_{j=1}^N\left(|y_4(t_{j+1})-y_4(t_j)|\right)\geq D_{H_2,H_3}.$$

This is our third such length, and we are done finding vertical lengths.  We need only find horizontal lengths now.  But this is done in exactly the same manner as in our search for vertical lengths.  We could even rotate our figures $90$ degrees either left or right, so that vertical lines become horizontal and vice versa, and perform the exact same proof as above.

We now construct our figure.  First, we have found a total length of 
$$2\cdot\left(D_{H_1,H_2}+D_{H_3,H_4}+D_{V_1,V_2}+D_{V_3,V_4}\right)+3\cdot\left(D_{h_2,H_3}+D_{V_2,V_3}\right).$$  
This gives us enough length to construct $A^{\square}$ and still have left over a total length of $D_{H_2,H_3}+D_{V_2,V_3}$.  With these lengths, we construct a box in the corner of $A^{\square}$ with dimensions $D_{H_2,H_3}\times D_{V_2,V_3}$.  This box, which we call $\tilde{B}$ in the corner has volume the same as $B^{\square}$, and therefore volume at least $\alpha$.  We can shrink it easily so that it has volume exactly $\alpha$, and by abuse of notation still call this possibly smaller rectangle $\tilde{B}$.  On the other hand, $A\cup B\subset A^{\square}$, and therefore $\mu(A^{\square})\geq1+\alpha$.  Therefore, $\mu(A^{\square}\setminus\tilde{B})\geq1+\alpha-\alpha$.  So we can easily move the sides of $A^{\square}$ that don't share joint boundary with $\tilde{B}$ inwards until the volume is exactly $1$.  This set we call $\tilde{A}$, and have completed our construction. 

Now, suppose that $b_{\rm bottom}=a_{\rm bottom}$, or $b_{\rm left}=a_{\rm left}$, etc.  That is one of the sides of $B^{\square}$ is contiguous with one of the lines of $A^{\square}$.  The process would be as above, except we would have $H_3=H_4$.  This means we wouldn't have to find two paths with vertical length $D_{H_3,H_4}$.  The rest of the proof would be the same.  Similarly, if two sides of $B^{\square}$ are contiguous with two sides of $A^{\square}$, say $a_{\rm bottom}=b_{\rm bottom}$ and $a_{\rm left}=b_{\rm left}$, then we wouldn't have to find paths with vertical length $D_{H_3,H_4}$ and we wouldn't have to find paths with horizontal length $D_{V_1,V_2}$.  The rest of the proof would be the same.

\end{proof}

The previous lemma took into account all of the cases when all four corners of $B^{\square}$ are contained in $A^{\square}$.  The only other two options are if two or one corner of $B^{\square}$ is contained in $A^{\square}$.

\begin{lemma}\label{lem:oneortwocorners}
If $(A,B)\in\gamma_{\alpha}$, and exactly one corner of $B^{\square}$ is contained in $A^{\square}$, then there exists $(\tilde{A},\tilde{B})\in\mathfs{F}_\alpha$, such that $\rdb(\tilde{A},\tilde{B})\leq\rdb(A,B)$.

\end{lemma}

\begin{proof} For the one corner case we argue that we can find sets, with a better double bubble perimeter and more joint volume than the original shapes, that looks like the general case of Figure \ref{fig:falpha}:  

\begin{figure}[h!]
\begin{tikzpicture}

\draw[blue, ultra thin] (6,0) to (8,0);
\draw[blue, ultra thin] (8,0) to (8,2);
\draw[blue, ultra thin] (8,2) to (7,2);
\draw[blue, ultra thin] (7,2) to (7,3);
\draw[blue, ultra thin] (7,3) to (6,3);
\draw[blue, ultra thin] (6,3) to (6,0);

\draw[red, ultra thin] (7,2) to (10,2);
\draw[red, ultra thin] (10,2) to (10,4);
\draw[red, ultra thin] (10,4) to (7,4);
\draw[red, ultra thin] (7,4) to (7,2);

\end{tikzpicture}\end{figure}

Here, the rectangle can be either $A^{\square}$ or $B^{\square}$, say $A^{\square}$, and the other set is $B^{\square}\setminus A^{\square}$.  Once we create these two sets, we are not necessarily done because the volumes may not be correct.  This may cause somewhat more of a problem than in previous cases, but in any case the method remains similar.

%
%
%
%
%
%

In this proof we are assuming that $\mu(A)=1$, and $\mu(B)=\alpha$. Since $A^{\square}$ and $B^{\square}$ only intersect in one corner, either $a_{\rm top}>b_{\rm top}$, or $b_{\rm top}>a_{\rm top}$.  We can suppose without loss of generality that $a_{\rm top}>b_{\rm top}$.  Let $H_1=\mathbb{R}\times\{a_{\rm top}\}$ be the horizontal line passing through $a_{\rm top}$, $H_2=\mathbb{R}\times\{b_{\rm top}\}$ be the horizontal line passing through $b_{\rm top}$, $H_3=\mathbb{R}\times\{a_{\rm bottom}\}$ be the horizontal line passing through $a_{\rm bottom}$, and $H_4=\mathbb{R}\times\{b_{\rm bottom}\}$ be the horizontal line passing through $b_{\rm bottom}$.  Note that the height of $A^{\square}\cap B^{\square}$ is the same as the distance between $H_2$ and $H_3$, which we will call $D_{H2,H3}$.  Furthermore, let the distance between $H_1$ and $H_2$ be $D_{H_1,H_2}$, and the distance between $H_3$ and $H_4$ be $D_{H_3,H_4}$. Now, let $V_1=\{b_{\rm left}\}\times\mathbb{R}$ be the vertical line passing through $b_{\rm left}$, $V_2=\{a_{\rm left}\}\times\mathbb{R}$ be the vertical line passing through $a_{\rm left}$, $V_3=\{b_{\rm right}\}\times\mathbb{R}$ be the vertical line passing through $b_{\rm right}$, and $V_4=\{a_{\rm right}\}\times\mathbb{R}$ be the vertical line passing through $a_{\rm right}$.  We define the distance between $V_1$ and $V_2$ to be $D_{V_1,V_2}$, the distance between $V_2$ and $V_3$ to be $D_{V_2,V_3}$, and the distance between $V_3$, and $V_4$ to be $D_{V_3,V_4}$.  See Figure \ref{fig:generalcase}.

\begin{figure}[h!]
\begin{center}
\begin{tikzpicture}

\draw[blue, thick] (0,0) to (0,3);
\draw[blue, thick] (0,3) to (2,3);
\draw[blue, thick] (2,3) to (2,2);
\draw[blue, thick] (2,2) to (3,2);
\draw[blue, thick] (3,2) to (3,2);
\draw[blue, thick] (3,2) to (3,0);
\draw[blue, thick] (3,0) to (0,0);


\draw[red, thick] (6,5) to (2,5);
\draw[red, thick] (2,5) to (2,2);
\draw[red, thick] (2,2) to (3,2);

\draw[red, thick] (3,2) to (6,2);
\draw[red, thick] (6,2) to (6,5);

\draw[dotted, thick] (0,-2) to (0,6);
\node (v1) at (-0.3,5.5) {$V_1$};

\draw[dotted, thick] (2,-2) to (2,6);
\node (v1) at (1.7,5.5) {$V_2$};

\draw[dotted, thick] (3,-2) to (3,6);
\node (v1) at (3.3,5.5) {$V_3$};

\draw[dotted, thick] (6,-2) to (6,6);
\node (v1) at (5.7,5.5) {$V_4$};

\draw[dotted, thick] (-2,5) to (7.4,5);
\node (v1) at (7,5.24) {$H_1$};

\draw[dotted, thick] (-2,3) to (7.4,3);
\node (v1) at (7,3.24) {$H_2$};

\draw[dotted, thick] (-2,2) to (7.4,2);
\node (v1) at (7,2.24) {$H_3$};

\draw[dotted, thick] (-2,0) to (7.4,0);
\node (v1) at (7,0.24) {$H_4$};

\draw[<->] (-1.9,3.1) to (-1.9,4.9);
\node (v1) at (-2.56,4.1) {$D_{H_1,H_2}$};

\draw[<->] (-1.9,2.1) to (-1.9,2.9);
\node (v1) at (-2.56,2.5) {$D_{H_2,H_3}$};

\draw[<->] (-1.9,0.1) to (-1.9,1.9);
\node (v1) at (-2.56,1) {$D_{H_3,H_4}$};

\draw[<->] (0.1,-1.9) to (1.9,-1.9);
\node (v1) at (1,-2.2) {$D_{V_1,V_2}$};

\draw[<->] (2.1,-1.9) to (2.9,-1.9);
\node (v1) at (2.54,-2.2) {$D_{V_2,V_3}$};

\draw[<->] (3.1,-1.9) to (5.9,-1.9);
\node (v1) at (4.5,-2.2) {$D_{V_3,V_4}$};

\end{tikzpicture}
\end{center}\caption{\label{fig:generalcase}}
\end{figure}
Notice that to construct $A^{\square}$ and $B^{\square}\setminus A^{\square}$ in this way we need two lengths of $D_{H_1,H_2}$, $D_{H_3,H_4}$, $D_{V_1,V_2}$, and $D_{V_3,V_4}$, as well as three lengths of $D_{H_2,H_3}$, and $D_{V_2,V_3}$.  In other words, for $\rho_{\text{DB}}((A^{\square},B^{\square}\setminus A^{\square}))$ to be at most $\rho_{\text{DB}}((A,B))$, we need to find two vertical lengths of $D_{H_1,H_2}$ in $\partial A$, two vertical lengths of $D_{H_3,H_4}$ in $\partial B$, and three vertical lengths of $D_{H_2,H_3}$ between $\partial A$ and $\partial B$.  Similarly for horizontal lengths.

First, there must be a point $p_1\in H_1\cap\partial A$, and another point $p_2\in H_3\cap\partial A$.  Between these two points there must be at least two disjoint paths, which we call $\lambda_i(t)=(x_i(t),y_i(t)):[0,1]\rightarrow\mathbb{R}^2$, $i\in\{1,2\}$, in $\partial A$, both of which have vertical length at least $D_{H_1,H_2}+D_{H_2,H_3}$.  That is, as before, for any $0\leq t_1\leq t_2\leq...\leq t_N\leq1$, 

$$\sum_{j=1}^N(|y_i(t_{j+1})-y_i(t_j)|)\geq D_{H_1,H_2}+D_{H_2,H_3}$$

Similarly in $\partial B$, we can find two disjoint paths of vertical length at least $D_{H_2,H_3}+D_{H_3,H_4}$.  Since there can be no joint boundary below $H_3$, the portions of these paths that measure at least $D_{H_3,H_4}$ have not been counted yet.  It remains to find a path in $\partial B$ with vertical length at least $D_{H_2,H_3}$ that we have yet to count. 

To find this path we define the infinite strip of height $D_{H_2,H_3}$, $S:=\mathbb{R}\times(a_{\rm bottom},b_{\rm top})$. If $\partial A$ has more than two crossings of $S$, then each of these crossings has vertical length of at least $D_{H_2,H_3}$, as above, and we are done.  So we may assume that $\partial A$ contains exactly two crossings of $S$ (it can't be less than two as we noted above).  Abbreviate these crossings $\xi_1, \xi_2$. In this case, consider $S\setminus \xi_1\cup\xi_2$.  This consists of three open sets, but only two are unbounded sets, and in one of these open sets that we find $B\cap S$.  We call these sets $S_1$ and $S_2$ such that $\xi_1\subset \partial S_1$ and $\xi_2\subset \partial S_2$. Suppose without loss of generality that $B\cap S\subset S_2$. There is at least one point $p_3\in\partial B\cap H_2$, and at least one point $p_4\in\partial B\cap H_3$, and there must be two distinct paths in $\partial B$ from $p_3$ to $p_4$.  We call these paths $\lambda_i(t)=(x_i(t),y_i(t)):[0,1]\rightarrow\mathbb{R}^2$, $i=3,4$. By planarity, only one of the paths, either $\lambda_3$ or $\lambda_4$ can have joint boundary with $\xi_2$, say $\lambda_3$.  This means that we have yet to include the vertical length of $\lambda_4$, and we have just found our third vertical length of $D_{H_2,H_3}$.  That is, for any $0\leq t_1\leq t_2\leq...\leq t_N\leq1$, 

$$\sum_{j=1}^N(|y_4(t_{j+1})-y_4(t_j)|)\geq D_{H_2,H_3}.$$

Finding the three horizontal lengths of $D_{V_2,V_3}$ follows the same argument.  The reason we can be sure that we won't double count anything is because, inside $A^{\square}\cap B^{\square}$ we have only counted vertical distance, and in the argument to find our three horizontal lengths of $D_{V_2,V_3}$ we would only count horizontal lengths.  So, we have proved that $\rho_{\text{DB}}(A^{\square},B^{\square})\leq\rho_{\text{DB}}(A,B)$.  It is clear that $\mu(A\cup B)\leq\mu(A^{\square}\cup B^{\square})$.  Since $A\subset A^{\square}$, $1=\mu(A)\leq\mu(A^{\square})$.  However, it is possible that $\mu(B^{\square}\setminus A^{\square})<\mu(B)=\alpha$.  This we must correct. For this purpose, let us refer to the notation we established in the introduction for the general case.


\begin{figure}[h!]
\begin{tikzpicture}[scale=0.80]
\draw[blue,   thick] (15,4) to (18,4); 
\draw (15.5,2) node{$a$};
\draw[blue,   thick] (15,0) to (18,0);
\draw (16.5,.5) node{$b$};
\draw[blue,   thick] (15,0) to (15,4);
\draw (19.25,1.5) node{$f$};
\draw[blue,   thick] (18,0) to (18,4);
\draw (16.5,4.5) node{$c$};
\draw (17,3.65) node{$d$};
\draw (17.7,2.5) node{$e$};
\node at (16.6,2) {$\textcolor{blue}{A^{\square}}$};
\node at (19,3) {$\textcolor{red}{B^{\square}\setminus A^{\square}}$};
\draw[red,   thick] (18,1) to (20,1);
\draw[red,   thick] (20,1) to (20,5); 
\draw[red,   thick] (20,5) to (16,5);
\draw[red,   thick] (16,5) to (16,4);

\end{tikzpicture}\caption{\label{fig:volumefix1}}
\end{figure}

We look at the two intervals of lengths $c,f$ and assume w.l.o.g that $f>c$ (see Figure \ref{fig:volumefix1}). Now we slide $B^{\square}\setminus A^{\square}$ down towards the longer edge $f$. By doing so we do not enlarge the boundary and we can only increase the area. 
There are two options now:
\begin{enumerate}
\item If $c+d+e\le a$ we get kissing rectangles where the rectangle on the left has greater area than $B^{\square}\setminus A^{\square}$.  (see Figure \ref{fig:volumefix2}).

Next we move area from $A^{\square}$ to the rectangle on the right until reaching the appropriate area $\alpha$. Consider $f'\ge f$ such that $\mu(\tilde{B})=f'\cdot (c+d+e)=\alpha$, and let $b'=b-(f'-f)$. Our final sets are of the general case class as can be seen in Figure \ref{fig:volumefix3}. Since we have only enlarged the total area we are guaranteed that $\mu(\tilde{A})\ge 1$. Reducing the area is easy and we have dealt with this before. 

\begin{figure}[H]
\begin{minipage}{.5\textwidth}
  \centering
\begin{tikzpicture}[scale=0.80]
\draw[blue,   thick] (15,4) to (18,4); 
\draw (15.3,2) node{\tiny{$a$}};
\draw[blue,   thick] (15,0) to (18,0);
\draw (16.5,.5) node{\tiny{$b$}};
\draw[blue,   thick] (15,0) to (15,4);
\draw (19.25,0.5) node{\tiny{$f$}};
\draw[blue,   thick] (18,0) to (18,4);

\draw (19.3,1.5) node{\tiny{$c+d+e$}};
\node at (16.6,2) {$\textcolor{blue}{A^{\square}}$};
\draw[red,   thick] (18,0) to (20,0);
\draw[red,   thick] (20,0) to (20,3); 
\draw[red,   thick] (20,3) to (18,3);

\end{tikzpicture}\caption{\label{fig:volumefix2}}
\end{minipage}%
\begin{minipage}{.5\textwidth}
  \centering
\begin{tikzpicture}[scale=0.80]
\draw[blue,   thick] (15,4) to (18,4); 
\draw (15.3,2) node{\tiny{$a$}};
\draw[blue,   thick] (15,0) to (18,0);
\draw (16,.5) node{\tiny{$b-(f'-f)$}};
\draw[blue,   thick] (15,0) to (15,4);
\draw (18.5,0.5) node{\tiny{$f'$}};
\draw[blue,   thick] (18,3) to (18,4);

\draw (19.3,1.5) node{\tiny{$c+d+e$}};
\node at (16.3,2) {$\textcolor{blue}{\tilde{A}}$};
\node at (18.4,2) {$\textcolor{red}{\tilde{B}}$};
\draw[red,   thick] (17,0) to (20,0);
\draw[red,   thick] (20,0) to (20,3); 
\draw[red,   thick] (20,3) to (17,3);
\draw[red,   thick] (17,0) to (17,3);

\end{tikzpicture}\caption{\label{fig:volumefix3}}
\end{minipage}
\end{figure}

\item If $c+d+e> a$ we get a general case type shape (see Figure \ref{fig:volumefix4}), where $d'=d-(a-e)$. If $(a+c)(f+d')\ge \alpha$, we can move the right edge of $A^{\square}$ to the left (decreasing $d'$) until getting $\mu(\tilde{B})=\alpha$. Since $\mu(\tilde{A}\cup\tilde{B})>1+\alpha$ we are guaranteed that $\mu(\tilde{A})\ge1$. We end up with a Figure much like Figure \ref{fig:volumefix4}.

If $(a+c)(f+d')< \alpha$  we take an $\tilde{f}$ such that $\tilde{f}\cdot (a+c)=\alpha$, then $\tilde{b}=b-(\tilde{f}-f)$. By taking $\tilde{B}$ to be the rectangle of side lengths $\tilde{f}$ and $(a+c)$ and $\tilde{A}$ to be the rectangle of side lengths $a$ and $\tilde{b}$ we only increased the total area thus we are guaranteed that $\mu(\tilde{A})\ge 1$. In this case we get kissing rectangles, see Figure \ref{fig:volumefix5}.
\begin{figure}[h!]
\begin{minipage}{.5\textwidth}
  \centering
\begin{tikzpicture}[scale=0.80]
\draw[blue,   thick] (15,4) to (18,4); 
\draw (15.5,2) node{$a$};
\draw[blue,   thick] (15,0) to (18,0);
\draw (16.5,.5) node{$b$};
\draw[blue,   thick] (15,0) to (15,4);
\draw (19.25,0.5) node{$f$};
\draw[blue,   thick] (18,0) to (18,4);
\draw (16.5,4.5) node{$c$};
\draw (17.5,3.65) node{$d'$};
\node at (16.6,2) {$\textcolor{blue}{A^{\square}}$};
\draw[red,   thick] (18,0) to (20,0);
\draw[red,   thick] (20,0) to (20,5); 
\draw[red,   thick] (20,5) to (17,5);
\draw[red,   thick] (17,5) to (17,4);
\end{tikzpicture}\caption{\label{fig:volumefix4}}
\end{minipage}%
\begin{minipage}{.5\textwidth}
  \centering
\begin{tikzpicture}[scale=0.80]
\draw[blue,   thick] (15,4) to (16.7,4); 
\draw (15.5,2) node{$a$};
\draw[blue,   thick] (15,0) to (18,0);
\draw (16,.5) node{$\tilde{b}$};
\draw[blue,   thick] (15,0) to (15,4);
\draw (18.25,0.5) node{$\tilde{f}$};
\draw (16.5,4.5) node{$c$};
\node at (16,2) {$\textcolor{blue}{\tilde{A}}$};
\node at (18.4,2) {$\textcolor{red}{\tilde{B}}$};
\draw[red,   thick] (16.7,0) to (20,0);
\draw[red,   thick] (20,0) to (20,5); 
\draw[red,   thick] (20,5) to (16.7,5);
\draw[red,   thick] (16.7,5) to (16.7,0);
\end{tikzpicture}\caption{\label{fig:volumefix5}}
\end{minipage}
\end{figure}

\end{enumerate}
 
\end{proof}

Now we need only deal with the case when two corners of $B^{\square}$ are in $A^{\square}$.

\begin{lemma}\label{lem:twocorners}
If $(A,B)\in\gamma_{\alpha}$, and two corners of $B^{\square}$ are contained in $A^{\square}$, then there exists $(\tilde{A},\tilde{B})\in\mathfs{F}_\alpha$, such that $\rdb(\tilde{A},\tilde{B})\leq\rdb(A,B)$.
\end{lemma}

\begin{proof}
Here we must have one of the following:  $b_{\rm top}>a_{\rm top}$, $b_{\rm right}>a_{\rm right}$, $b_{\rm bottom}<a_{\rm bottom}$, or $b_{\rm left}<a_{\rm left}$.  Let's suppose, without loss of generality, that $b_{\rm right}>a_{\rm right}$.  We construct our configuration in a similar way as before.  First, since two corners of $B^{\square}$ are in $A^{\square}$, it follows that $a_{\rm left}<b_{\rm left}\leq a_{\rm right}$.  However, if $b_{\rm left}=a_{\rm right}$, we can just replace $A$ with $A^{\square}$ and $B$ with $B^{\square}$.  This will increase the volume of both $A$ and $B$, and increase their joint boundary.  Then we can reduce the volumes as necessary, which will only decrease the double bubble perimeter.  So, we can assume that $a_{\rm left}<b_{\rm left}<a_{\rm right}$.  Let $H_1$ be the horizontal line passing through $a_{\rm top}$, $H_2$ be the horizontal line passing through $b_{\rm top}$, $H_3$ be the horizontal line passing through $b_{\rm bottom}$, and $H_4$ the horizontal line passing through $a_{\rm bottom}$.  We define $D_{H_i,H_{i+1}}$ as before, $i=1,2,3$.  Similarly define $V_1,V_2,V_3,V_4$ as the vertical lines passing through $a_{\rm left},b_{\rm left},a_{\rm right}$, and $b_{\rm right}$, respectively.  Let $D_{V_1,V_{i+1}}$ be defined as before, $i=1,2,3$.  Notice that we haven't eliminated the possibility that $H_1=H_2$, or $H_3=H_4$, or both.  See Figure \ref{fig:embedrec}. 

\begin{figure}[h!]
\begin{tikzpicture}[scale=0.6]


\node at (5.5,5.5) {$A^{\square}$};
\draw[red, thick] (0,0) to (0,6);
\draw[red, thick] (0,6) to (6,6);
\draw[red, thick] (6,6) to (6,0);
\draw[red, thick] (6,0) to (0,0);


\node at (7.5,3.5) {$B^{\square}$};
\draw[blue, thick] (4,4) to (8,4);
\draw[blue, thick] (8,4) to (8,2);
\draw[blue, thick] (8,2) to (4,2);
\draw[blue, thick] (4,2) to (4,4);


\draw[dotted, thick] (-3,6) to (9,6);
\node at (8.9,6.3) {$H_1$};


\draw[dotted, thick] (-3,4) to (9,4);
\node at (8.9,4.3) {$H_2$};


\draw[dotted, thick] (-3,2) to (9,2);
\node at (8.9,2.3) {$H_3$};


\draw[dotted, thick] (-3,0) to (9,0);
\node at (8.9,0.3) {$H_4$};


\draw[dotted, thick] (0,-3) to (0,9);
\node at (0.45,8.8) {$V_1$};


\draw[dotted, thick] (4,-3) to (4,9);
\node at (4.45,8.8) {$V_2$};


\draw[dotted, thick] (6,-3) to (6,9);
\node at (6.45,8.8) {$V_3$};


\draw[dotted, thick] (8,-3) to (8,9);
\node at (8.45,8.8) {$V_4$};


\draw[<->, thick] (-2.96,4.02) to (-2.96,5.98);
\node at (-4,5) {$D_{H_1,H_2}$};


\draw[<->, thick] (-2.96,2.02) to (-2.96,3.98);
\node at (-4,3) {$D_{H_2,H_3}$};


\draw[<->, thick] (-2.96,0.02) to (-2.96,1.98);
\node at (-4,1) {$D_{H_3,H_4}$};


\draw[<->, thick] (0.02,-2.96) to (3.98,-2.96);
\node at (2,-3.5) {$D_{V_1,V_2}$};


\draw[<->, thick] (4.02,-2.96) to (5.98,-2.96);
\node at (5.1,-3.5) {$D_{V_2,V_3}$};


\draw[<->, thick] (6.02,-2.96) to (7.98,-2.96);
\node at (7.1,-3.5) {$D_{V_3,V_4}$};

\end{tikzpicture}\\
\caption{\label{fig:embedrec}}
\end{figure}

We wish to find two disjoint paths with vertical lengths at least $D_{H_1,H_2}$, two disjoint paths with vertical length as least as long as $D_{H_3,H_4}$, three disjoint paths with vertical lengths at least as long as $D_{H_2,H_3}$, two disjoint paths with horizontal length at least as long as $D_{V_1,V_2}$, three disjoint paths with horizontal lengths at least as long as $D_{V_2,V_3}$, and two disjoint paths with horizontal lengths at least as long as $D_{V_3,V_4}$.  We achieve this much the same as in the previous Lemma. 


Finding the horizontal lengths is nearly identical. We can then proceed to construct our sets. If move $B^{\square}$ up until $H_1=H_2$ we reduce the double bubble perimeter and we get a shape of the general case type. Now we can fix the volumes in the same way as in the previous Lemma. 
\end{proof}
\section{KKT analysis}\label{sec:KKT}

In the previous section, we constructed a finite list of set types $\mathfs{F}_\alpha$ in which from any configuration $(A,B)\in\gamma_{\alpha}$ we obtain a configuration $(\tilde{A},\tilde{B})\in \mathfs{F}_{\alpha}$ so that $\rho_{\text{DB}}((\tilde{A},\tilde{B}))\le \rho_{\text{DB}}((A,B))$. In the upcoming analysis, however, it is convenient to note that all the cases in $\mathfs{F}_\alpha$ can be represented by the $6$ parameter configuration we called ``General Case" given in Figure \ref{fig:analysisgeneral}.

\begin{figure}[!h]
\centering
\begin{tikzpicture}[scale=0.4]

\draw[blue,  thick] (7,4) to (10,4); 
\draw (7.5,2) node{$a$};
\draw[blue,  thick] (7,0) to (10,0);
\draw (8.5,.5) node{$b$};
\draw[blue,  thick] (7,0) to (7,4);
\draw (11.25,1.5) node{$f$};
\draw[blue, thick] (10,0) to (10,4);
\draw (8.5,4.5) node{$c$};
\draw (9,3.65) node{$d$};
\draw (9.7,2.5) node{$e$};
\draw[red,  thick] (10,1) to (12,1);
\draw[red,  thick] (12,1) to (12,5); 
\draw[red, thick] (12,5) to (8,5);
\draw[red,  thick] (8,5) to (8,4);
\draw (9,8.25) node{${\tiny \begin{array}{l} \rdb=2(a+b+c+f)+(d+e)\\ ab=\textrm{1 or }\alpha\\ b\ge d,a\ge e\\ cd+cf+ef=\alpha\textrm{ or 1}\\ a,b,c,d,e,f>0 \end{array}}$};
\end{tikzpicture}\caption{\label{fig:analysisgeneral}}
\end{figure}
The other cases, kissing rectangles for instance, occur when some of the parameters defining the geometry of the configuration are set to zero. For kissing rectangles, this would consist of setting $e$ and $f$ to zero (or $c=d=0$) in Figure \ref{fig:analysisgeneral}. Using this configuration we have made a geometric problem into the minimization of a hyper-plane with algebraic inequality constraints. We want to minimize the $\rho_{\text{DB}}$ in Figure\ref{fig:analysisgeneral}, while remaining within the inequality contraints.

The Karush Kuhn Tucker method \cite{karush} is one method used for calculating minima in such problems; with some care the global minimum of each element of the above list can be found. In the previous lemmas, it was sometimes ambiguous which of the two shapes in a configuration had unit volume and which $\alpha$. For our purposes, that just means we must alternate which shape has volume $\alpha$ and analyze both. 

%
%
%


$\\$

The original problem we would have to solve is 6 dimensional, but our two equality constraints reduce it to four. These are that $ab=\alpha\textrm{ or 1}$ and $cd+cf+ef=\textrm{1 or }\alpha$. Using these constraints, our problem becomes to minimize

$$2(a+\frac{\alpha\textrm{ or 1}}{a}+c+\frac{(\textrm{1 or }\alpha)-cd}{c+e})+(d+e),$$ subject to 

$$(\alpha\textrm{ or 1})\ge ad,a\ge e,a\ge 0,c\ge 0,d\ge 0,e\ge 0,(\textrm{1 or }\alpha)\ge cd.$$

The variable $a$ cannot be zero, so we can exclude its inequality lagrange multiplier. The $(\alpha\textrm{ or 1})\ge ad$ constraint comes from $b=\frac{\alpha\textrm{ or 1}}{a}\ge d$ and $(\textrm{1 or }\alpha)\ge cd$ from $\frac{(\textrm{1 or }\alpha)-cd}{c+e}=f\ge 0$. Note that the conventional conditions required for KKT (equality constraints be affine, etc.) are not satisfied here. In this case, however, we can constrain our variable space in the upper octant and inside some hypercube. This is because if one of the variables becomes larger than some example double bubble perimeter, it is not optimal.  A final precaution we address is that due to the exclusion of the constraint $a\ge 0$ from our list of inequality constraints under analysis, our space of variables is not truly closed and not truly compact. In the upcoming analysis, other portions of the variable space will similarly be excluded (it will be explicitly mentioned when a portion of the variable space is excluded) because they are not valid configurations. However at some finite distance in variable space from these coordinates, the fact that $a$ (as an example) becomes small forces another dimension to grow as $a^{-1}$. This growth eventually pushes $\rho_{DB}$ over the perimeter of that example configuration we used to build the hypercube. Therefore in a similar manner we bound our domain in these degenerate cases by curves a finite distance from the degenerate case; the boundary here doesn't require explicit checking because by construction it is too large.

Thus our modified variable space is truly closed and bounded and therefore compact. The global minimum is either along the boundary or is the lowest local minimum inside the domain itself. The KKT method checks all of this by implementing the inequality constraints which define the boundary. 

Our above $\rho_{DB}$ yeilds the lagrangian

\begin{align*}
L=&2(a+\frac{\alpha\textrm{ or 1}}{a}+c+\frac{(\textrm{1 or }\alpha)-cd}{c+e})+(d+e)+((\alpha\textrm{ or 1})-ad)\mu_{1}+(a-e)\mu_{2}\\
&+c\mu_{3}+d\mu_{4}+e\mu_{5}+((\textrm{1 or }\alpha)-cd)\mu_{6}.
\end{align*}

Lets call $\beta=\alpha\textrm{ or 1},\gamma=\textrm{1 or }\alpha$, then the gradient of $L$ becomes

\begin{align*}
&\Bigg(2(1-\frac{\beta}{a^{2}})-d\mu_{1}+\mu_{2},2(1-\frac{\gamma+ed}{(c+e)^{2}})+\mu_{3}-d\mu_{6},2(\frac{-c}{c+e})+1+\mu_{4}-c\mu_{6},...\\
&...2(-\frac{\gamma-cd}{(c+e)^{2}})+1-a\mu_{1}-\mu_{2}+\mu_{5}\Bigg).
\end{align*}

For the KKT method, the inequality lagrangian multipliers are either positive and their conditions applied (i.e. if $\mu_{1}>0$ then $\beta=ad$) or they are zero. This means without symmetry arguments that there are $2^{6}$ systems of nonlinear equations to check for minima. If $\mu_{6}>0$ then $cd=\gamma$, then from our second equality constraint we have that $(c+e)f+cd=\textrm{1 or }\alpha$ and $(c+e)f=0$, or $c=e=0$ or $f=0$. The first is impossible and the second becomes kissing rectangles. As one may see, in parts of this calculation we will eliminate some of the 64 possible systems of equations by demonstrating geometrically what they resolve to, and then doing that case only once. In this case, this simple calculation got rid of 31 systems and left us only needing to calculate the general kissing rectangles case. Similarly if $\mu_{3}>0$, $\mu_{4}>0$, or $\mu_{5}>0$ then $c=0$, $d=0$, or $e=0$ and again we have kissing rectangles. This means we only have to calculate it with $\mu_{1}$ and $\mu_{2}$ possibly positive, leaving 4 systems of equations and the kissing rectangles to analyze. If we remember that $\mu_{1}$ corresponds to $b\ge d$, then by symmetry we realize that $\mu_{1}$ and $\mu_{2}$ being activated alone have the same effect by symmetry of the figure, so we need only check $\mu_{1}>0$ and both $\mu_{1}>0$ and $\mu_{2}>0$. So we have only 3 systems and kissing rectangles to calculate.

\subsection{Kissing rectangles} For kissing rectangles we let the sides of one rectangle be $a$,$b$, and the other $c$,$d$ (notation is changed for convenience, see Figure \ref{fig:falpha}). So $ab=\beta$ while $cd=\gamma$. And taking the $b$-$d$ edge to be the kissing side, we then take $b$ to be the larger size ($b\ge d$ or $\frac{\beta}{a}\ge \frac{\gamma}{c}$). Our lagrangian becomes

$$L=2(a+\frac{\beta}{a}+c)+\frac{\gamma}{c}+\mu_{1}(c\beta-a\gamma)$$

where the inequality constraint $\mu_{1}$ is $b\ge d$ or in terms of $a$ and $c$ $c\beta\ge a\gamma$. The gradient is

$$\left(2(1-\frac{\beta}{a^{2}})-\gamma\mu_{1},2-\frac{\gamma}{c^{2}}+\mu_{1}\beta\right)=0.$$

The unconstrained minimum has $a=\sqrt{\beta}=b, c=\sqrt{\frac{\gamma}{2}}$, $d=
\sqrt{2\gamma}$ and $$\rdb=2\left(2\sqrt{\beta}+\sqrt{\frac{\gamma}{2}}\right)+\sqrt{2\gamma}.$$

The perimeter becomes either $4\sqrt{\alpha}+2\sqrt{2}$ or $4+2\sqrt{2\alpha}$. This first equation again has $\gamma=1$, $\beta=\alpha$, $a=\sqrt{\alpha}=b$, and $d=\sqrt{2}$. Now we need $b\ge d$, which would mean $\alpha\ge 2$. This is a contradiction, because $\alpha\le 1$. The second equation, in the exact same way, forces us instead to have $\sqrt{\alpha}\le \sqrt{1/2}\implies \alpha\le 1/2$. For $\alpha>1/2$ we need to apply the constraint. We know now that $\gamma=\alpha,\beta=1$, so $c=a\alpha$, and our gradient becomes

$$\left((2-\frac{2}{a^{2}})-\alpha\mu_{1},2-\frac{\alpha}{c^{2}}+\mu_{1}\right)=0.$$

Multiplying the second equation in the gradient by $\alpha$ and adding it to the first we have $2\alpha-\frac{1}{a^{2}}+2-\frac{2}{a^{2}}=0$, so $2(\alpha+1)=\frac{3}{a^{2}},a=\sqrt{\frac{3}{2(\alpha+1)}},c=a\alpha$. Plugging these into our equation for $\rho_{\text{DB}}$ we obtain the kissing rectangle solution

\begin{lemma}\label{lem:kissingrecsolution}
For any $\alpha\in(0,1]$ the infimum among all $(A,B)\in\mathfs{F}_\alpha$ of the kissing rectangles type is achieved and admits:
$$\inf\{\rdb(A,B):(A,B)\in\mathfs{F}_\alpha\text{ are kissing rectangles}\}=(4+2\sqrt{2\alpha})\ind_{\left(0,\frac{1}{2}\right)}+(2\sqrt{6(1+\alpha)})\ind_{[\frac{1}{2},1]}$$
\end{lemma}

\subsection{Embedded rectangles} Now we go on to the case of $\min(\mu_{1},\mu_{2})>0$. This means $b=d$ and $a=e$ which becomes an ``Embedded rectangle" type. Again it becomes useful to breifly depart from our original notation. We can label this structure with $a$,$b$ being the inner rectangle's sides, $c$ and $d$ the outer (See Figure \ref{fig:falpha}). This means $ab=\beta$ the volume of the inner rectangle, $cd=\gamma+\beta$ the volume of the inner rectangle and the outer peice, and our lagrangian becomes

$$2(c+d)+a+b+\lambda_{1}(ab-\beta)+\lambda_{2}(cd-(\gamma+\beta))$$

or in an unrestrained form

$$2(\frac{\beta+\gamma}{d}+d)+a+\frac{\beta}{a}$$

The gradient of this is

$$\left(1-\frac{\beta}{a^{2}},1-\frac{\beta+\gamma}{d^{2}}\right),$$ so $a=b=\sqrt{\beta},c=d=\sqrt{\gamma+\beta}$, and altogether the perimeter is

$$\rho_{\text{DB}}=2\sqrt{\beta}+4\sqrt{\gamma+\beta},$$ which is either $2\sqrt{\alpha}+4\sqrt{1+\alpha}$ or $2+4\sqrt{1+\alpha}$. This second one is too large and is proven by our example in the paper to be suboptimal. So for this case we have

\begin{lemma}\label{lem:embedmin}
For any $\alpha\in(0,1]$ the infimum among all $(A,B)\in\mathfs{F}_\alpha$ of the embedded rectangles type is achieved and admits:
$$\inf\{\rdb(A,B):(A,B)\in\mathfs{F}_\alpha\text{ are embedded rectangles}\}=2\sqrt{\alpha}+4\sqrt{1+\alpha}.$$
\end{lemma}

\subsection{General case} We now check the unconstrained case where all $\mu_{i}=0$. For this purpose we go back to the original notation we established.

Assume all $\mu_{i}=0$. Then we have the following gradient

$$\left(2(1-\frac{\beta}{a^{2}}),2(1-\frac{\gamma+ed}{(c+e)^{2}}),2(\frac{-c}{c+e})+1,2(-\frac{\gamma-cd}{(c+e)^{2}})+1\right)=0$$

So $a=\sqrt{\beta},c=e,d=\frac{\gamma}{3c}$. The second constraint gives us $(c+e)^{2}=\gamma+ed$. Plugging in our expression for $e$ and $d$ in terms of $c$, we have $c=\sqrt{\frac{\gamma}{3}}$.

Now that we know $a$ and $c$ in terms of $\beta$ and $\gamma$, we can obtain the perimeter from the first expression in the KKT analysis (i.e. the perimeter for the general case in terms of our variables).

We obtain the following as perimeter: $2(2\sqrt{\beta}+2\sqrt{\frac{\gamma}{3}})+2\sqrt{\frac{\gamma}{3}}=4\sqrt{\beta}+6\sqrt{\frac{\gamma}{3}}$, which is $4\sqrt{\alpha}+6\sqrt{\frac{1}{3}}$ or $4+6\sqrt{\frac{\alpha}{3}}$. For the perimeter to represent a valid shape, we need $b\ge d\implies \sqrt{\alpha}\ge\sqrt{3}$, which is never true. The second double bubble perimeter is never optimal as can be seen by comparing it to the double bubble perimeter of the sets in Theorem \ref{thm:maintheorem} part \ref{part:exampleshapes} .


Now to the last case where only $\mu_{1}>0$. So $\beta=ad$ and we have the following gradient

$$\left(2(1-\frac{\beta}{a^{2}})-d\mu_{1},2(1-\frac{\gamma+ed}{(c+e)^{2}}),2(\frac{-c}{c+e})+1,2(-\frac{\gamma-cd}{(c+e)^{2}})+1-a\mu_{1}\right)=0$$

So $c=e$ and this gradient becomes (by reducing the dimension and removing e)

$$(2(1-\frac{\beta}{a^{2}})-d\mu_{1},2(1-\frac{\gamma+cd}{4c^{2}}),2(-\frac{\gamma-cd}{4c^{2}})+1-a\mu_{1})=0$$

Multiply the first equation by $a$, the third by $d$, and subtract the third by first to get $2(a-\frac{\beta}{a}+\frac{\gamma d-cd^{2}}{4c^{2}})-d$. Next from the second equation we get $4c^{2}=\gamma+cd$, so $d=\frac{\gamma d+cd^{2}}{4c^{2}}$. Applying this to the derived equation, we get $d=2(\frac{\beta}{a}-a)$. Now from our $\mu_{1}$ constraint we know $ad=\beta$ or $2(\beta-a^{2})=\beta$, so we get $a=\sqrt{\frac{\beta}{2}}$ and $d=\sqrt{2\beta}$. Now plugging $d$ back into $4c^{2}=\gamma+cd$ we get $4c^{2}-c\sqrt{2\beta}-\gamma=0$. So from the quadratic equation we have $c=e=\frac{\sqrt{2\beta}+\sqrt{2\beta+16\gamma}}{8}$ (the minus solution is negative so it can't work). This obtains all the relevant variables. From this we can plug into the perimeter and obtain:

$\rho_{\text{DB}}=2\left(\sqrt{\frac{\beta}{2}}+\frac{\beta}{\sqrt{\frac{\beta}{2}}}+\frac{\sqrt{2\beta}+\sqrt{2\beta+16\gamma}}{8}+\frac{\gamma-\frac{\sqrt{2\beta}+\sqrt{2\beta+16\gamma}}{8}\sqrt{2\beta}}{2\frac{\sqrt{2\beta}+\sqrt{2\beta+16\gamma}}{8}}\right)+\left(\sqrt{2\beta}+\frac{\sqrt{2\beta}+\sqrt{2\beta+16\gamma}}{8}\right)$

For this if let $\gamma=\alpha$, it is always more than double bubble perimeter of the sets in Theorem \ref{thm:maintheorem} part \ref{part:exampleshapes}. If $\beta=\alpha$ and $\gamma=1$, then we do get valid and smaller answers for $\alpha<0.12$. But we also need $e\le a$, or $\frac{\sqrt{2\alpha}+\sqrt{2\alpha+16}}{8}\le \sqrt{\frac{\alpha}{2}}$ which is only true outside of the range from 0 to 1. Therefore this is an invalid answer and we obtain that: 
\begin{lemma}\label{lem:generalmin}
The minimum of the double bubble perimeter over every configuration in $\mathfs{F}_\alpha$ is the minimum between the kissing rectangles and embedded rectangle given in Lemmas \ref{lem:kissingrecsolution} and \ref{lem:embedmin}. 
\end{lemma}

\begin{remark}
The two graphs for $Vol(A)=1,Vol(B)=\alpha$ and $Vol(A)=\alpha,Vol(B)=1$ are different because in one case the encased rectangle has volume 1 and never gets small enough to be absorbed into the bigger rectangle like in the embedded rectangle case, so we only see the kissing rectangles case for it. The other case exhibits all portions of the minimum. See figure \ref{fig:volumecompare} for a comparison of the two cases.
\begin{figure}[h!]
\centering
\begin{tikzpicture}[scale=0.9]
	\begin{axis}[
		legend pos=south east,
		xlabel=$\alpha$,
		ylabel={$\rho_{\text{DB}}(\alpha)$}
	]
	\addplot[color=blue] coordinates {
		(0,4)
		(.0125,4.2485292)
		(.025,4.3659191)
		(.05,4.5459939)
		(.1,4.8276909)
		(.15,5.0641188)
		(.2,5.264911)
		(.25,5.4142136)
		(.3,5.5491933)
		(.35,5.6733201)
		(.4,5.7888544)
		(.45,5.8973666)
		(.5,6)
		(.55,6.0991803)
		(.6,6.1967734)
		(.65,6.2928531)
		(.7,6.3874878)
		(.75,6.4807407)
		(.8,6.5726707)
		(.85,6.6633325)
		(.9,6.7527772)
		(.95,6.8410526)
		(1.,6.9282032)
		
	};
	\addlegendentry{Vol(A)=$\alpha$,Vol(B)=1}
	\addplot[color=red] coordinates {
		(0,4)
		(.0125,4.3162278)
		(.025,4.4472136)
		(.05,4.6324555)
		(.1,4.8944272)
		(.15,5.0954451)
		(.2,5.264911)
		(.25,5.4142136)
		(.3,5.5491933)
		(.35,5.6733201)
		(.4,5.7888544)
		(.45,5.8973666)
		(.5,6)
		(.55,6.0991803)
		(.6,6.1967734)
		(.65,6.2928531)
		(.7,6.3874878)
		(.75,6.4807407)
		(.8,6.5726707)
		(.85,6.6633325)
		(.9,6.7527772)
		(.95,6.8410526)
		(1.,6.9282032)
	};
	\addlegendentry{Vol(A)=1,Vol(B)=$\alpha$}
	\end{axis}
\end{tikzpicture}\caption{\label{fig:volumecompare}}
\end{figure}
\end{remark}


\section{Proof of Theorem \ref{thm:maintheorem}}\label{sec:proof}
In this section we collect the results of the previous sections to prove our main result.

In Lemmas \ref{lem:kissingrecsolution}, \ref{lem:embedmin} and \ref{lem:generalmin} we analyze each of the possible elements of $\mathfs{F}_\alpha$ using the KKT method and by comparing them we can find a global minimizer in $\mathfs{F}_\alpha$ which we call here $\chi_\alpha\in\mathfs{F}_\alpha$ satisfying:
\begin{lemma}\label{lem:chi}
\begin{enumerate}[I.]
\item \label{lem:chi1}For any $0<\alpha \le1$ there is a $\chi_\alpha\in\mathfs{F}_\alpha$ such that for any $(A,B)\in\mathfs{F}_\alpha$, $\rdb(\chi_\alpha)\le\rdb((A,B))$. 
\item \label{lem:chi2}\begin{align*}
\rho_{\text{DB}}(\chi_\alpha)&=(4\sqrt{1+\alpha}+2\sqrt{\alpha})\ind_{[0,\frac{688-480\sqrt{2}}{49}]}(\alpha)+(4+2\sqrt{2\alpha})\ind_{\left(\frac{688-480\sqrt{2}}{49},\frac{1}{2}\right)}(\alpha)\\ \nonumber &+(2\sqrt{6(1+\alpha)})\ind_{[\frac{1}{2},1]}(\alpha)
\end{align*}
\item \label{lem:chi3}For $\alpha=\frac{688-480\sqrt{2}}{49}$ we can choose for $\chi_\alpha$ either Figure \ref{fig:possibleinfchi} (a) or (b), for $\alpha\in[1/2,1]$ $\chi_\alpha$ satisfies Figure \ref{fig:possibleinfchi} (c), for $\alpha\in\left(\frac{688-480\sqrt{2}}{49},\frac{1}{2}\right)~$ $\chi_\alpha$ satisfies Figure \ref{fig:possibleinf} (b), and for $\alpha\in\left(0,\frac{688-480\sqrt{2}}{49}\right)~$ $\chi_\alpha$ satisfies Figure \ref{fig:possibleinfchi} (a).
\begin{figure}[h!]\begin{tikzpicture}[scale=0.7]

\draw[blue,  thin] (0,4) to (4,4);
\draw[blue,  thin] (0,0) to (4,0);
\draw[blue,  thin] (0,0) to (0,4);
\draw[blue,  thin] (4,0) to (4,4);
\draw[blue, thin] (3,3) to (3,4);
\draw[blue, thin] (3,3) to (4,3);
\draw (2,0.35) node{$\sqrt{1+\alpha}$};
\draw (0.85,2) node{$\sqrt{1+\alpha}$};
\draw (2.5,3.5) node{$\sqrt{\alpha}$};
\draw (3.5,2.5) node{$\sqrt{\alpha}$};
\draw (2,-0.5) node{(a)};
\draw[blue,  thin] (5,4) to (9,4);
\draw[blue,  thin] (5,0) to (9,0);
\draw[blue,  thin] (5,0) to (5,4);
\draw[blue,  thin] (9,0) to (9,4);
\draw (7,0.25) node{$1$};
\draw (5.25,2) node{$1$};
\draw (7,-0.5) node{(b)};
\draw[blue,  thin] (9,1) to (10.5,1);
\draw[blue,  thin] (9,3) to (10.5,3);
\draw[blue,  thin] (10.5,1) to (10.5,3);
\draw (9.80,1.75) node{$\sqrt{2\alpha}$};
\draw (9.80,3.5) node{$\frac{\sqrt{2\alpha}}{2}$};
\draw[blue,  thin] (11,4) to (14.5,4);
\draw[blue,  thin] (11,0) to (14.5,0);
\draw[blue,  thin] (11,0) to (11,4);
\draw[blue,  thin] (14,0) to (14,4);
\draw[blue,  thin] (14.5,4) to (16.5,4);
\draw[blue,  thin] (14.5,0) to (16.5,0);
\draw[blue,  thin] (16.5,4) to (16.5,0);
\draw (11.99,2) node{$\sqrt{\frac{2(1+\alpha)}{3}}$};
\draw (12.2,3.3) node{$\sqrt{\frac{3}{2(1+\alpha)}}$};
\draw (15.22,3.3) node{$\alpha\sqrt{\frac{3}{2(1+\alpha)}}$};
\draw (14,-0.5) node{(c)};

\end{tikzpicture}\caption{\label{fig:possibleinfchi}}
\end{figure}
\end{enumerate}
\end{lemma}

\begin{proof}[Proof of Theorem \ref{thm:maintheorem}]
First let $0<\alpha\le 1$. Take some sequence $\chi_i\in\gamma_\alpha$ such that $\lim_{i\rightarrow\infty}\rho_{\text{DB}}(\chi_i)=\rho_{\text{DB}}(\Gamma_\alpha)$. By Lemmas \ref{lem:containedrec}, \ref{lem:oneortwocorners} and \ref{lem:twocorners}, for each element of this sequence we obtain an element $\tilde{\chi}_{i}\in\mathfs{F}_{\alpha}$ such that $\rho_{\text{DB}}(\tilde{\chi}_{i})\le \rho_{\text{DB}}(\chi_{i})$. 
By Lemma \ref{lem:chi} part \ref{lem:chi1} there is a $\chi_\alpha\in\mathfs{F}_\alpha$ satisfying for any $i\in\BN$,
$$\rdb(\chi_\alpha)\le\rdb(\tilde{\chi}_i)\le\rdb(\chi_i).$$ 
We have that 
$$\rdb(\Gamma_\alpha)\le\rdb(\chi_\alpha)\le \lim_{i\rightarrow\infty}\rho_{\text{DB}}(\chi_i)=\rdb(\Gamma_\alpha),$$
and thus $\chi_\alpha\in\Gamma_\alpha$ and $\Gamma_\alpha$ is non empty, establishing part \ref{existencepart} of Theorem \ref{thm:maintheorem}. Since we have that $\chi_\alpha\in\Gamma_\alpha$, Lemma \ref{lem:chi} parts \ref{lem:chi2} and \ref{lem:chi3} establishes Theorem \ref{thm:maintheorem} parts \ref{part1}  and \ref{part:exampleshapes}.

\end{proof}

\bibliography{ri}

\begin{thebibliography}{10}

\bibitem{alexander1990wulff}
K.~Alexander, J.T. Chayes, and L.~Chayes.
\newblock The wulff construction and asymptotics of the finite cluster
  distribution for two-dimensional bernoulli percolation.
\newblock {\em Communications in mathematical physics}, 131(1):1--50, 1990.

\bibitem{biskup2015isoperimetry}
M.~Biskup, O.~Louidor, E.~B. Procaccia, and R.~Rosenthal.
\newblock Isoperimetry in two-dimensional percolation.
\newblock {\em Communications on Pure and Applied Mathematics},
  68(9):1483--1531, 2015.

\bibitem{bodineau2000rigorous}
T.~Bodineau, D.~Ioffe, and Y.~Velenik.
\newblock Rigorous probabilistic analysis of equilibrium crystal shapes.
\newblock {\em Journal of Mathematical Physics}, 41(3):1033--1098, 2000.

\bibitem{cerf2006wulff}
R.~Cerf.
\newblock {\em The Wulff Crystal in Ising and Percolation Models: Ecole
  D'Et{\'e} de Probabilit{\'e}s de Saint-Flour XXXIV-2004}.
\newblock Springer, 2006.

\bibitem{foisy1993standard}
J.~Foisy, Manuel Alfaro~G., J.~Brock, N.~Hodges, and J.~Zimba.
\newblock The standard double soap bubble in r2 uniquely minimizes perimeter.
\newblock {\em Pacific journal of mathematics}, 159(1):47--59, 1993.

\bibitem{doublebubbleconj}
M.~Hutchings, F.~Morgan, M.~Ritoré, and A.~Ros.
\newblock Proof of the double bubble conjecture.
\newblock {\em Annals of Mathematics}, 155(2):459--489, 2002.

\bibitem{karush}
W.~Karush.
\newblock Minima of functions of several variables with inequalities as side
  conditions.
\newblock {\em Master thesis, University of Chicago}, 1939.

\bibitem{milman2018gaussian}
E.~Milman and J.~Neeman.
\newblock The gaussian double-bubble conjecture.
\newblock {\em arXiv preprint arXiv:1801.09296}, 2018.

\bibitem{morgan1998wulff}
F.~Morgan, C.~French, and S.~Greenleaf.
\newblock Wulff clusters in $r^2$.
\newblock {\em The Journal of Geometric Analysis}, 8(1):97, 1998.

\bibitem{taylor1974existence}
J.~Taylor.
\newblock Existence and structure of solutions to a class of nonelliptic
  variational problems.
\newblock In {\em Symposia Mathematica}, volume~14, pages 499--508, 1974.

\bibitem{taylor1975unique}
J.~Taylor.
\newblock Unique structure of solutions to a class of nonelliptic variational
  problems.
\newblock In {\em Proc. Symp. Pure Math. AMS}, volume~27, pages 419--427, 1975.

\bibitem{wul1901frage}
G.~Wul.
\newblock Zur frage der geschwindigkeit des wachstums und der auflosung der
  kristall achen.
\newblock {\em Z. Kristallogr}, 34:449--530, 1901.

\end{thebibliography}
\bibliographystyle{plain}

%
%



\end{document}